\documentclass[11pt,reqno]{amsart}
\usepackage[utf8]{inputenc}
\usepackage[english]{babel}
\usepackage{amssymb}
\usepackage{amsmath}
\usepackage{amsthm}
\usepackage{amstext}
\usepackage{hyperref}
\usepackage{enumerate}
\usepackage{indentfirst}
\usepackage{xcolor}
\usepackage{cleveref}
\usepackage{mathtools}
\crefname{section}{Section}{Sections}
\crefname{equation}{Equation}{Equation}
\crefname{lemma}{Lemma}{Lemmata}
\crefname{bem}{Remark}{Remarks}
\crefname{kor}{Corollary}{Corollaries}
\crefname{defin}{Definition}{Definitions}
\crefname{prop}{Proposition}{Propositions}
\crefname{folg}{Conclusion}{Conclusions}
\crefname{bsp}{Example}{Examples}
\crefname{claim}{Claim}{Claims}
\newtheorem{thm}{Proposition}[section]
\newtheorem{theorem}[thm]{Theorem}

\newtheorem{lem}[thm]{Lemma}

\newtheorem{claim}[thm]{Claim}
\theoremstyle{definition}
\newtheorem{defi}[thm]{Definition}
\newtheorem{bei}[thm]{Example}
\newtheorem{bem}[thm]{Remark}
\newcommand{\R}[0]{{\mathbb R}}

\DeclareMathOperator{\Hom}{Hom}

\DeclareMathOperator{\cx}{cx}

\title{Complexity of gauge bounded Cartier algebras}
\begin{document}

\author[H.~July]{Henry July}
\address{Bergische Universit\"at Wuppertal, Fakult\"at 4, Gau\ss stra\ss e 20, 42119 Wuppertal, Germany}
\email{\href{mailto:hjuly@uni-wuppertal.de}{hjuly@uni-wuppertal.de}}
\author[A.~St\"abler]{Axel St\"abler}
\address{Johannes Gutenberg-Universit\"at Mainz\\ Fachbereich 08\\ Staudingerweg 9\\ 55099 Mainz\\Germany} 
\email{\href{mailto:staebler@uni-mainz.de}{staebler@uni-mainz.de}}
\thanks{We thank Manuel Blickle for useful comments on an earlier draft.}
\begin{abstract}
We show that a gauge bounded Cartier algebra has finite complexity. We also give an example showing that the converse does not hold in general.
\end{abstract}
\maketitle
\vspace{-0.61em}

\section{Introduction}
The central notion of this note is the following

\begin{defi}
Let $R$ be an $F$-finite noetherian (commutative) ring.
\begin{enumerate}[(i)]
\item We denote by \[\mathcal{C}^R=\bigoplus_{e\geq 0}\mathcal{C}^R_e=\bigoplus_{e\geq 0}\Hom_R(F^e_\ast R,R) \] the \emph{total Cartier algebra} of $R$. We note that this is a non-commutative ring, where multiplication $\varphi \cdot \psi$ is defined as $\varphi \circ F_\ast^a \psi$ for $\varphi \in \mathcal{C}_a^R$ and $\psi \in \mathcal{C}_b^R$. Also note that any $\mathcal{C}_e^R$ is an $R$-bimodule where the left module structure is the ordinary one and the right module structure is given by noting that $F_\ast^e R = R$ as rings. These module structures are then related by $\varphi \cdot r^{p^e} = r \cdot \varphi$ for any $\varphi \in \mathcal{C}_e^R$.
\item We call a graded subring $\mathcal{D}\subseteq \mathcal{C}^R$ a Cartier subalgebra of $R$, if $\mathcal{D}_0=R$ and $\mathcal{D}_e\neq 0$ holds for some $e>0$. We write $\mathcal{D}_+$ for $ \bigoplus_{e \geq 1} \mathcal{D}_e$.
\end{enumerate}
\end{defi}

Given an $F$-finite field $k$ we fix an isomorphism $k \to F^! k = \Hom_k(F_\ast k, k)$ (where on the right we consider the right $k$-module structure). Denote the adjoint map $F_\ast k \to k$ by $\lambda$. If $R = k[x_1,\ldots, x_n]$ is a polynomial ring over an $F$-finite field, then $\Hom_R(F_\ast^e R, R)$ is generated as a right-$R$-module by the so-called Cartier operator  $\kappa^e\colon F_\ast^e R \to R$
\[ \kappa^e(c x^\alpha) = \begin{cases} \lambda^e(c), & \text{ if } \alpha_i = p^e -1 \text { for all } 1 \leq i \leq n, \\ 0, &\text{ if } \alpha_i < p^e -1 \text{ for some } i. \end{cases}\]
Little is lost if the reader assumes that $k$ is perfect and $\lambda$ is defined as the $p$th root map. In particular, in the case that $R$ is a polynomial ring, $\mathcal{C}^R$ is generated by $\kappa$ as an $R$-algebra.

It is understood that $\mathcal{C}^R$ is not finitely generated in many cases of interest if $R$ is not a Gorenstein ring (\cite{katzmannonfg}, \cite{MR2905024}). Due to this, weaker finiteness notions have been introduced. The first one is that of the \emph{complexity} of the Cartier algebra which loosely speaking measures how fast the number of generators as a right $R$-module grows with the homogeneous degree.

\begin{defi}[{see \cite{enescuyaofcomplexity}}]
Let $(R, \mathcal{D})$ be as above.
\begin{enumerate}[(i)]
\item For $e \geq -1$ let $G_e \coloneqq G_e(\mathcal{D})$ be the subring of $\mathcal{D}$ generated by elements of degree $\leq e$. We write $k_e$ for the minimal number of homogeneous generators of $G_e$ and call $(k_e - k_{e-1})_e$ the \emph{complexity sequence} of $\mathcal{D}$.
\item Provided that the $k_e$ are finite the complexity of $\mathcal{D}$ is defined as \begin{align*}
	\cx(\mathcal{D})\coloneqq\inf\{n\in \R_{>0}:k_e-k_{e-1}=\mathcal{O}(n^e)\}.
	\end{align*} We follow the convention that the infimum over $\varnothing$ is $\infty$.
	\item The \emph{Frobenius exponent} of $\mathcal{D}$ is defined as $\exp_F(\mathcal{D}) = \log_p(cx(\mathcal{D}))$.
\end{enumerate}

\end{defi}

The second finiteness notion is that of \emph{gauge boundedness} (\cite{blickletestidealsvia}). This notion is very useful in proving various desired properties for test ideal filtrations. From now on we let $S = k[x_1, \ldots, x_n]$ be a polynomial ring in $n$ variables over an $F$-finite field $k$.

\begin{defi}
\label{gaugebound}
\begin{enumerate}[(i)]
\item For each $d \geq 0$ let $S_d$ be the $k$-subspace of $S$ which is generated by all monomials $x_1^{\alpha_1}\cdots x_n^{\alpha_n}$ such that $0\leq \alpha_1,...,\alpha_n\leq d$. We will loosely refer to this as the \emph{maximum norm}. Let $I \subseteq S$ be an ideal and $R = S/I$. The $S_d$ induce an increasing filtration on $R$ by setting $R_{- \infty} = 0$ and $R_d = S_d \cdot 1_R$ for $d \geq 0$. For $r \in R$ we define a \emph{gauge} $\delta\colon R \to \mathbb{N} \cup \{- \infty\}$ via
\begin{align*}
\delta(r)=\begin{cases}
\ -\infty, &\textit{if}\ r=0,\\ \ d,\ &\textit{if}\ r\neq 0\ \textit{and}\ r\in R_d\smallsetminus R_{d-1}.\end{cases}
\end{align*}
\item Let $\mathcal{D} \subseteq \mathcal{C}^R$ be a Cartier algebra. We say that $\mathcal{D}$ (or the pair $(R, \mathcal{D})$) is \emph{gauge bounded} if there is a set \[\{ \psi_i \, \vert \, \psi_i \in \mathcal{D}_{e_i}, e_i \geq 1 \}\] which generates $\mathcal{D}_{+}$ as a right $R$-module and a constant $K$ such that
\[\delta(\psi_i(r)) \leq \frac{\delta(r)}{p^{e_i}} + K \] for all $r \in R$.
\end{enumerate}
\end{defi}

Overall, the notion of gauge boundedness is still poorly understood. While there are examples of non-gauge bounded Cartier algebras (\cite{blicklestaeblerfuntestmodules}), it is unknown whether the full Cartier algebra $\mathcal{C}^R$ for a quotient $R$ of $S$ is always gauge bounded or not.

Our goal in this note is to show that gauge boundedness implies finite complexity and that the converse is not true. We also give an elementary proof of a result of Enescu--Perez (\cite[Theorem 3.9]{enescuperezfexponent}) at the expense of a worse bound.
\section{The result}
Recall that throughout $S = k[x_1,\ldots, x_n]$ is a polynomial ring over an $F$-finite field $k$.

When dealing with $\mathcal{C}^R$ for $R = S/I$ we have by Fedder (\cite{fedderisom}) an isomorphism of $R$-modules
\[\Psi\colon \frac{(I^{[p^e]}:I)}{I^{[p^e]}} \longrightarrow \Hom_R(F_\ast^e R, R). \]
In particular, for a Cartier subalgebra $\mathcal{D} \subseteq \mathcal{C}^R$ we may also study the action of the Cartier algebra \[\bigoplus_{e \geq 0} \kappa^e J_e \] on $R$, where $\kappa$ is a generator of $\Hom_S(F_\ast S, S)$ and $J_e = \Psi^{-1}(\mathcal{D}_e)$. 

\begin{lem}
\label{keylemma}
Let $I \subset S$ be an ideal, $R = S/I$ and $\mathcal{D} \subseteq \mathcal{C}^R$ a subalgebra on $R$. For each $e$ we fix a minimal system of generators $(f_1, \ldots, f_{a_e})$ of $J_e$ and denote the maximum of the $\deg f_i$ by $d(J_e)$. If $d(J_e) \leq K p^{te}$ for some constants $K$ and $t$, then the complexity sequence $(k_e - k_{e-1})$ is in $\mathcal{O}(p^{etn})$.
\end{lem}
\begin{proof}
We have the following inequality where for the first equality we simply count the number of monomials of degree $\leq d$ in $S$
\begin{align*}
    k_{e}-k_{e-1}&\leq k_e\leq \sum_{d=0}^{Kp^{te}}\dim_kS_d =\binom{n+Kp^{te}}{Kp^{te}}\\&=\frac{(n+Kp^{te})\cdots(Kp^{te}+1)}{n(n-1)\cdots 1}\\&=\left(1+Kp^{te}\right)\left(1+\frac{Kp^{te}}{n-1}\right)\cdots \left(1+Kp^{te}\right)\\&\leq (Kp^{te}+1)^n\\& \leq K'p^{te n}
\end{align*}
for some constant $K'$. Thus $(k_e - k_{e-1}) \in \mathcal{O}(p^{etn})$ as claimed.
\end{proof}

We can now give an elementary proof of \cite[Theorem 3.9]{enescuperezfexponent} with a worse bound (they achieve $\dim R$ instead of $\dim S = n$).

\begin{theorem}
\label{elementaryproof}Let $I \subset S$ be an ideal, $R = S/I$ and $\mathcal{D} \subseteq \mathcal{C}^R$ a subalgebra on $R$. With notation as above, if $d(J_e)=\mathcal{O}(p^{te})$ holds for some fixed $t$ (and some choice of generating system), then 
\begin{align*}
    \exp_F(\mathcal{D})\leq t\cdot n.
\end{align*}
\end{theorem}
\begin{proof}
From \Cref{keylemma} we get $k_e-k_{e-1}=\mathcal{O}(p^{ten})$ and therefore also $\cx(\mathcal{D})\leq p^{tn}$. Hence, $\exp_F(\mathcal{D})=\log_p(\cx(\mathcal{D}))\leq \log_p(p^{tn})=tn$.
\end{proof}

\begin{bem}
\begin{enumerate}[(a)]
\item Note that \cite[Theorem 3.9]{enescuperezfexponent} requires the ideal to be homogeneous whereas we do not need this assumption. Alternatively, one can also avoid this by arguing as in \cite[Remark 2.3]{MR3192605} to obtain the bound $t (\dim R +1)$ via \cite[Theorem 3.9]{enescuperezfexponent}.
\item The strategy to apply \Cref{elementaryproof} above is to use \cite[Lemma 3.2, Theorem 3.3]{MR3192605} as outlined in \cite{enescuperezfexponent}. Unfortunately, the proof of \cite[Lemma 3.2]{MR3192605} contains a gap. Namely, it is not clear why equation (5) on page $3525$ holds. Hence, it is at present not known whether there are interesting cases where we can apply \Cref{elementaryproof}.
\end{enumerate}
\end{bem}

\begin{theorem}
Let $I \subseteq S$ be an ideal. If $(R=S/I,\mathcal{D})$ is gauge bounded, then $\cx(\mathcal{D}) \leq p^n$.
\end{theorem}
\begin{proof}
First, we assert
\begin{claim}
\label{claim1}
There are generators $f_1, \ldots, f_m$ of $J_e$ such that $\delta(f_i) \leq K p^e$ where $K$ is a constant independent of $e$.
\end{claim}

Assuming this claim we proceed as follows. Using the bound obtained from \Cref{claim1} and an argument just as in \Cref{keylemma} (with the maximum norm instead of the degree) yields

\[ k_e -k_{e-1} \leq K' p^{en}\] for some constant $K'$. Hence, $\cx(\mathcal{D})\leq p^n$.

It remains to prove \cref{claim1}:
We fix a generating system $\{\psi_{\gamma}\}$ as in \Cref{gaugebound} and $e \geq 1$. Further, choose $\psi_1, \ldots, \psi_r$ in our generating system which generate $\mathcal{D}_e$. As explained before \Cref{keylemma} $\psi_i$ is of the form $\kappa^e f_i$ and the $f_i$ form a system of generators of $J_e$. Our goal is to bound the gauges of these $f_i$. Therefore fix one and omit the index.

Let $\alpha$ be a multiindex such that $c x^\alpha$ is a monomial of $f$ such that $\delta(f) = \delta(c x^{\alpha})$. Write $\alpha = p^e r + \alpha'$ for unique $r, \alpha' \in \mathbb{N}^n$ with $0 \leq \alpha'_i \leq p^e -1$.
Then
\begin{align*}
\kappa^e\left(c x^\alpha x^{p^e - 1 - \alpha'}\right) = \kappa^e \left( c x^{p^e r + p^e - 1}\right) = x^r \kappa^e\left(c x^{p^e -1}\right) = x^r \kappa^e(c)
\end{align*}
and therefore
\[ \delta\left(\kappa^e\left(c x^\alpha x^{p^e - 1 - \alpha'}\right)\right) = \max \{ r_1, \ldots, r_n\} \eqqcolon r_\mathrm{max}.\]
On the other hand, we have
\begin{align*}
\delta \left( \kappa^e \left( cx^{\alpha}  x^{p^e -1 -\alpha'} \right) \right) = \delta\left( \kappa^e \left(f x^{p^e -1 -\alpha'}\right) \right) &= \delta \left( \psi_i\left(x^{p^e -1 - \alpha'}\right)\right)\\ &\leq \frac{\delta\big( x^{p^e-1 - \alpha'} \big)}{p^e} + K\\ &\leq \frac{p^e -1}{p^e} +K \leq 1 + K,
\end{align*}
where the first inequality holds due to \Cref{gaugebound} (ii).

Hence, $r_\mathrm{max} \leq 1 + K$ which is independent of $e$. Thus all $\alpha'_i$ and $r_i p^{e}$ are in $\mathcal{O}(p^e)$. Hence also all $\alpha_i$ which proves \cref{claim1}.
\end{proof}

We now come to the promised example showing that there are Cartier algebras of finite complexity which are not gauge bounded.

\begin{bei}
Let $R=k[x,y]$ for an $F$-finite field $k$ of prime characteristic $p$. It is verified in \cite[Example 7.13]{blicklestaeblerfuntestmodules} that 
\[ \mathcal{D} \coloneqq \bigoplus_{e \geq 0} \kappa^e \mathfrak{a}_e\] with $\mathfrak{a}_0 = R$ and $\mathfrak{a}_e = (x^2, xy^{ep^e})$ for $e \geq 1$ is a Cartier algebra which is not gauge bounded. Note that in \cite{blicklestaeblerfuntestmodules} the gauge is induced by the degree, but we may just as well use the one induced by the maximum norm as in \Cref{gaugebound} (i). However, the complexity of $\mathcal{D}$ is clearly one.
\end{bei}

\bibliographystyle{alpha}
\bibliography{Bibfile}
\end{document}